\documentclass{lmsRarXiv}
\usepackage{amsmath,amssymb,amsfonts,latexsym}

\newtheorem{theorem}{Theorem}[section]

\newtheorem{proposition}[theorem]{Proposition}

\newnumbered{assertion}{Assertion}
\newnumbered{conjecture}{Conjecture}
\newnumbered{definition}{Definition}
\newnumbered{hypothesis}{Hypothesis}
\newnumbered{remark}{Remark}
\newnumbered{note}{Note}
\newnumbered{observation}{Observation}
\newnumbered{problem}{Problem}
\newnumbered{question}{Question}
\newnumbered{algorithm}{Algorithm}
\newnumbered{example}{Example}
\newunnumbered{notation}{Notation}

\simpleequations

\title[Even-freeness of cyclic $2$-designs]%
 {Even-freeness of cyclic $2$-designs} %

\author{Yuichiro Fujiwara}

\classno{05B05 (primary), 05E18, 94B25 (secondary)}

\extraline{This work is supported by JSPS.}

\begin{document}
\maketitle

\begin{abstract}
A Steiner $2$-design of block size $k$ is an ordered pair $(V, \mathcal{B})$ of finite sets such that
$\mathcal{B}$ is a family of $k$-subsets of $V$ in which each pair of elements of $V$ appears exactly once.
A Steiner $2$-design is said to be $r$-even-free if for every positive integer $i \leq r$ it contains no set of $i$ elements of $\mathcal{B}$
in which each element of $V$ appears exactly even times.
We study the even-freeness of a Steiner $2$-design when the cyclic group acts regularly on $V$.
We prove the existence of infinitely many nontrivial Steiner $2$-designs of large block size
which have the cyclic automorphisms and higher even-freeness than the trivial lower bound
but are not the points and lines of projective geometry.
\end{abstract}

\section{Introduction}
\noindent
A \textit{set system} is an ordered pair $(V, \mathcal{B})$ of finite sets, where $\mathcal{B}$ is a family of subsets of $V$.
The elements of $V$ and those of $\mathcal{B}$ are called \textit{points} and \textit{blocks} respectively.
A \textit{packing} of \textit{order} $v$ and \textit{block size} $k$ is a set system such that
the cardinality $\vert V \vert$ of the point set is $v$,
each block in $\mathcal{B}$ is a $k$-subset of $V$,
and every pair of points appear at most once in a block in $\mathcal{B}$.
A \textit{Steiner} $2$-\textit{design} is a special packing, where every pair of points appear exactly once in a block in $\mathcal{B}$.
In the language of combinatorial design theory, these two kinds of object are a $2$-$(v,k,1)$ packing and an $S(2,k,v)$ respectively.
To avoid triviality, we generally assume that $k \geq 3$ except when blocks of size two play an important role.

Two set systems $(V, \mathcal{B})$ and $(V', \mathcal{B}')$ are \textit{isomorphic} if there is a bijection
$\tau : V \rightarrow V'$ such that $\{\{\tau(b_i) \ \vert \ 0 \leq i \leq k-1\} \ \vert\ \{b_i \ \vert \ 0 \leq i \leq k-1\} \in \mathcal{B}\} = \mathcal{B}'$.
We use the term \textit{automorphism group} to mean a subgroup of the full automorphism group
formed by the set of all isomorphisms $V \rightarrow V$ that map $(V, \mathcal{B})$ to itself.

Automorphism groups of set systems have attracted attention in various research
from classical studies such as the classification of finite simple groups
to much younger fields in algebraic combinatorics.
Since the advent of digital revolution,
the study of automorphism groups of set systems has become relevant to computer science and engineering as well.

One particularly important case is when the cyclic group acts regularly on the points of a set system, namely a \textit{cyclic} set system.
Take a subset $B = \{b_0, \dots, b_{k-1}\}$ of integers and define $B + i {\pmod v}= \{b_0 + i {\pmod v},\dots, b_{k-1} + i {\pmod v}\}$.
By taking $V = \{0, 1, \dots, v-1\}$, a cyclic set system on $v$ points can be understood as the kind of set system $(V, \mathcal{B})$
in which for any $B \in \mathcal{B}$ we have $B + 1{\pmod v} \in \mathcal{B}$.
In what follows, we always assume that a cyclic set system on $v$ points is written this way with its point set $V$ being $\textit{\textbf{Z}}_v$ unless otherwise stated.

This simple description of a cyclic set system has allowed for many discoveries of cyclic set systems hidden in problems in disciplines
outside of the theories of finite groups and combinatorial designs.
Cyclic Steiner $2$-designs are particularly versatile in this regard. Their applications have been found
in many areas such as perpendicular magnetic recording in engineering (see Vasic, Kurtas, and Kuznetsov \cite{VKK}),
perfect secrecy authentication in cryptography (see Huber \cite{Huber}),
quasi-cyclic low-density parity-check (LDPC) codes in digital communications (see Lan \textsl{et al}.\ \cite{LTLMH}),
and quantum packet framing in quantum information science (see the author \cite{Fujiwara1}), to name a few.
A cyclic $2$-$(v,k,1)$ packing is also an interesting object, especially when it has the largest possible number of blocks for given $v$ and $k$;
such a packing is \textit{maximum}.
In fact, maximum cyclic packings are equivalent to optimal optical orthogonal codes (OOCs) of index one in coding theory (see Fuji-Hara and Miao \cite{FM}).
Since cyclic Steiner $2$-designs are automatically maximum packings, they are quintessential optimal OOCs as well.

While algebraic properties of set systems are of importance and interest from both the purely mathematical and more practical viewpoints,
some other properties have also gained equally great attention in various fields.
Avoidance of prescribed small set systems inside a Steiner $2$-design is an example of such properties.

A set system $(W, \mathcal{C})$ is called a \textit{configuration} of another set system $(V, \mathcal{B})$ if $W \subseteq V$ and $\mathcal{C} \subseteq \mathcal{B}$.
A configuration is \textit{even} if for any point $a \in W$ the number of blocks in $\mathcal{C}$ that contain $a$ is even.
A set system is $r$-\textit{even-free} if for any positive integer $i \leq r$ it contains no even configurations that consist of $i$ blocks.
Trivially an $r$-even-free set system is also an $(r-1)$-even-free for $r \geq 2$.
Because no pair of points appear more than once in a Steiner $2$-design,
it is straightforward to see that every $S(2,k,v)$ is $k$-even-free.
$S(2,k,v)$s with even-freeness higher than this trivial lower bound, especially the case when $k=3$ and $r=4$, have been investigated from various points of view
(see Colbourn and Rosa \cite{TRIPLESYSTEMS} and Colbourn and the author \cite{CF} for known results on $r$-even-free $S(2,3,v)$s in combinatorial design theory).

Research on avoidance problems typically draws on techniques in extremal set theory and combinatorial design theory.
The last two decades have seen rapid progress in the latter kind, partly due to the development of modern combinatorial design theory
(see, for example, Ling, Colbourn, Grannell, and Griggs \cite{ANTIPASCHJLMS} for a landmark advancement in techniques for constructing $4$-even-free $S(2,3,v)$s),
and partly due to the flood of discoveries of applications of $r$-even-free $S(2,k,v)$s with large $r$,
(for some recent examples of applications, see
the author and Colbourn \cite{FC} for data compaction for scan testing of very-large-scale integrated (VLSI) circuits,
Chee, Colbourn, and Ling \cite{RAIDERASURECOLBOURN} for efficient codes for redundant arrays of independent disks (RAID),
and the author \textsl{et al}.\ \cite{FCVBT} and the author and Tonchev \cite{FT}
for quantum error correction requiring only one Einstein-Podolsky-Rosen pair of maximally entangled quantum states).

Because of the applications of finite groups to combinatorial design theory,
it is becoming increasingly more intriguing to investigate automorphism groups of $S(2,k,v)$s in the context of avoidance problems.
In fact, recent important results in this area often exploit group actions to give constructions for combinatorial designs in avoidance problems
(see, for instance, Forbes, Grannell, and Griggs \cite{FGG1}).
The author \cite{sparseYF} investigated unavoidable configurations
when an $S(2,3,v)$ possesses automorphism groups that are frequently exploited
in the context of a long-standing avoidance problem, called Erd\H{o}s' $r$-sparse conjecture,
showing that many typical automorphism groups severely limit the range of avoidable configurations.

In light of the recent development of both purely mathematical and applicational aspects in this area,
it might be natural to study the relation of the important avoidance property, that is, even-freeness, to fundamental automorphism groups of combinatorial designs.
The primary purpose of the present paper is to investigate the existence problem of highly even-free packings
that admit abelian groups acting regularly on their points.
We are particularly interested in explicit constructions for cyclic $r$-even-free $S(2,k,v)$s with $r$ higher than $k$, the trivial lower bound.

In the following section we briefly review known relevant facts in combinatorial design theory and present some immediate results
on cyclic $r$-even-free designs. Section \ref{rconst} gives recursive constructions for cyclic $(k+1)$-even-free $S(2,k,v)$s
which work for a wide range of $k$.
Concluding remarks are given in Section \ref{conclusion}.

\section{Preliminaries}\label{prep}
\noindent
We first review basic facts in combinatorial design theory.
For proofs and more comprehensive treatments of combinatorial design theory,
the interested reader is referred to Beth, Jungnickel, and Lenz \cite{BJL} and Colbourn and Rosa \cite{TRIPLESYSTEMS}.

As defined in the previous section,
a Steiner $2$-design of order $v$ and block size $k$, briefly $S(2,k,v)$,
is an ordered pair $(V, \mathcal{B})$ of finite sets such that every pair of points in $V$ appear exactly once in a block in $\mathcal{B}$.
When $v \leq k$, an $S(2,k,v)$ is \textit{trivial}.
A simple counting argument shows that an $S(2,k,v)$ exists only if $v-1 \equiv 0 \pmod{k-1}$ and $v(v-1) \equiv 0 \pmod{k(k-1)}$.
It is known that these necessary conditions are asymptotically sufficient in the sense that
there exists a constant $v_k$ that depends only on $k$
such that for every $v > v_k$ satisfying the necessary conditions there exists an $S(2,k,v)$ (see Wilson \cite{W1,W2,W3}).

The necessary conditions for the existence become simpler if we assume that the cyclic group of order $v$ acts regularly on the points of an $S(2,k,v)$,
that is, when the Steiner $2$-design is cyclic. In this case, it is not difficult to see that a cyclic $S(2,k,v)$ exits only if $v \equiv 1,\ k \pmod{k(k-1)}$.

Proving nontrivial sufficient conditions seems challenging.
The existence problem of a cyclic $S(2,3,v)$ is famously known as Heffter's difference problem,
which was posed by Heffter \cite{Heffter,Heffter2} in 1896 and 1897.
The problem was solved in the affirmative during the last century with the definite exception of $v = 9$ (see Colbourn and Rosa \cite{TRIPLESYSTEMS}).
For block size $k \geq 4$, however, very little is known about the existence of cyclic Steiner $2$-designs despite its long history;
indeed, there are not many known systematic constructions (see, for example, Chen and Wei \cite{CW} for the case $k = 4$).

The knowledge of the existence of an $r$-even-free $S(2,k,v)$ found in the literature is also much more complete for the case when $k=3$
while the cases of larger block sizes remain nearly completely open.
The \textit{Pasch} is the unique possible even configuration on four blocks in an $S(2,3,v)$:
$\{\{a,b,c\}, \{a,d,e\}, \{b,d,f\}, \{c,e,f\}\}$.
The \textit{generalized Pasch} of block size $k$ is the smallest even configuration that may exist in an $S(2,k,v)$.
By considering the fact that no pair of points occur twice in a block in a Steiner $2$-design,
it is easy to see that for each block size $k$ the generalized Pasch consists of $k+1$ blocks and is unique up to isomorphism.
Hence, an $S(2,k,v)$ is $(k+1)$-even-free if and only if it contains no Pasch configurations for $k=3$
or generalized Pasch configurations for other values of $k$.
Note that if the block size is odd, a $2i$-even-free Steiner $2$-design is automatically $(2i+1)$-even-free
because there are no even configurations on an odd number of blocks of odd block size.
Hence, an $S(2,3,v)$ avoiding Pasch configurations, for example, is not only $4$-even-free, but also $5$-even-free.
While the number of nonisomorphic $S(2,k,v)$s grows exponentially as $v$ becomes larger,
those avoiding Pasch or generalized Pasch configurations appear to be rare.
In fact, an exhaustive computational result by Kaski and \"{O}sterg\r{a}rd \cite{KO} showed that,
of all $11, 084, 874, 829$ $S(2,3,19)$s, only $2, 591$ are $5$-even-free.
It is notable that only two of them are both cyclic and $5$-even-free.

Extensive investigations have been carried out for the case when $k=3$.
The complete spectrum of those $v$ for which there exists a $5$-even-free $S(2,3,v)$ is now known:
\begin{theorem}[(Grannell, Griggs, and Whitehead \cite{GGW})]\label{anti-Pasch}
There exists a $5$-even-free $S(2,3,v)$ if and only if $v \equiv 1, 3 \pmod{6}$ except $v= 7, 13$.
\end{theorem}

Finite geometries are the sources of the known $S(2,k,v)$s of larger $k$ with high even-freeness.
\begin{theorem}[(The author \textsl{et al}.\ \cite{FCVBT})]\label{PG}
For any integer $m \geq 3$ and any odd prime power $q$,
the points and lines of projective geometry $\textup{PG}(m,q)$ form a $(2q+1)$-even-free $S(2,q+1,\frac{q^{m+1}-1}{q-1})$.
\end{theorem}
\begin{theorem}[(Frumkin and Yakir \cite{FY})]\label{PG2}
For any odd prime power $q$
there exists a $(q^2+q)$-even-free $S(2,q+1,q^2+q+1)$, which is a projective plane.
\end{theorem}
\begin{theorem}[(M\"{u}ller and Jimbo \cite{MJ})]\label{AG}
For any integer $m \geq 2$ and any odd prime power $q$,
the points and lines of affine geometry $\textup{AG}(m,q)$ form a $(2q-1)$-even-free $S(2,q,q^m)$.
\end{theorem}

The following is the sharpest known upper bound on the even-freeness of an $S(2,3,v)$:
\begin{theorem}[(The author and Colbourn \cite{FC})]\label{upperbound}
There exists no nontrivial $8$-even-free $S(2,3,v)$.
\end{theorem}
It is conjectured that there are no nontrivial $6$-even-free $S(2,3,v)$s (see Colbourn and the author \cite{CF}).

While the conjecture is made through purely combinatorial observations,
it is not difficult to see that it holds if we limit ourselves to cyclic $S(2,3,v)$s. In fact, the same can be proved
for $S(2,k,v)$s with point-transitive, abelian automorphisms:
\begin{proposition}\label{transitive}
If an abelian automorphism group acts transitively on the point set of a nontrivial $r$-even-free $S(2,k,v)$ that is not a projective plane,
then $r \leq 2k-1$.
\end{proposition}
\begin{proof}
Let $(V, \mathcal{B})$ be an $S(2,k,v)$ with $v > k$.
Assume that an abelian automorphism group $G$ acts transitively on $V$.
If $\mathcal{B} = Orb_G(B)$ for some block $B \in \mathcal{B}$,
then the $S(2,k,v)$ is either trivial or an $S(2,k,k(k-1)+1)$, which is a projective plane.
Hence, we assume that there are two distinct block orbits.
Take two blocks $B_0$ and $B_1$ that share a point $a$, where $Orb_G(B_0) \not= Orb_G(B_1)$.
Write $B_0 = \{a, a^g, \dots, a^{g_{k-1}}\}$ and $B_1 = \{a, a^{g_k}, \dots, a^{g_{2(k-1)}}\}$, where $g_i \in G$ for $0 \leq i \leq 2(k-1)$ are distinct.
By developing these two blocks, we obtain a configuration $C = \{B_0, B_0^{g_k}, \dots, B_0^{g_{2(k-1)}}, B_1, B_1^g, \dots, B_1^{g_{k-1}}\}$.
Because $Orb_G(B_0) \not= Orb_G(B_1)$ and a pair of points occur exactly once in a block in $\mathcal{B}$,
we have the cardinality $\vert C \vert = 2k$.
Every point appearing in $C$ occurs exactly twice in a block in $C$. Thus, it is an even configuration on $2k$ blocks.
\end{proof}

The above upper bound can be met by infinitely many Steiner $2$-designs admitting such automorphism groups.
For instance, a Steiner $2$-design formed by the points and lines of projective geometry PG$(m,q)$ admits a Singer cycle.
Hence, the $S(2,k,v)$s in Theorem \ref{PG} are all cyclic, having the highest achievable even-freeness for a nontrivial abelian point-transitive $S(2,k,v)$.
\begin{proposition}\label{cyclicPG}
For any integer $m \geq 3$ and any odd prime power $q$ there exists a cyclic $(2q+1)$-even-free $S(2,q+1,\frac{q^{m+1}-1}{q-1})$.
\end{proposition}
\begin{proposition}\label{cyclicPG2}
For any odd prime power $q$ there exists a cyclic $(q^2+q)$-even-free $S(2,q+1,q^2+q+1)$.
\end{proposition}
The points and lines of PG$(m,q)$ for $q$ even also give a cyclic $S(2,q+1,\frac{q^{m+1}-1}{q-1})$.
However, they have even configurations on $q+2$ blocks,
which results in the lowest possible even-freeness for a Steiner $2$-design (see the author \textsl{et al}.\ \cite{FCVBT}).

Because the elementary abelian group acts transitively on the points of the the $S(2,k,v)$ formed by the points and lines of $\textup{AG}(m,q)$,
Theorem \ref{AG} gives another infinite series of Steiner $2$-designs meeting the upper bound.
A notable observation is that the $S(2,k,v)$ based on $\textup{AG}(m,q)$ admits a cycle of length $v-1$ as an automorphism fixing one point.
Taking the point set $V = \{\infty\}\cup\textit{\textbf{Z}}_{v-1}$,
the blocks containing the fixed point form a single orbit $Orb_{\tau}(B)$, where $B = \{\infty\}\cup\{\frac{i(v-1)}{k-1}\ \vert \ 0 \leq i \leq k-2\}$ and $\tau : i \mapsto i+1$.
Thus, by discarding this orbit, we obtain a maximum cyclic packing of order $v-1$ with the same high even-freeness.
\begin{proposition}\label{cyclicAG}
For any odd prime power $q$ there exists a maximum cyclic $(2q-1)$-even-free packing of order $q^m-1$ and block size $q$.
\end{proposition}
As is the case with projective geometry, it is known that Steiner $2$-designs based on AG$(m,q)$ with $q$ even have the lowest possible even-freeness.

For $k=3$, Brouwer \cite{ANTIPASCHBROUWER} and Doyen \cite{Doyen} proved that the well-known Bose construction generates a $5$-even-free $S(2,3,v)$
for all orders $v \equiv 3 \pmod{6}$ that are not divisible by seven.
With a little restriction, this can be made into a construction for cyclic $5$-even-free $S(2,3,v)$s.
\begin{proposition}\label{bose}
There exists a cyclic $5$-even-free $S(2,3,v)$ for all $v \equiv 3 \pmod{6}$ that are not divisible by $7$ or $9$.
\end{proposition}
\begin{proof}
Take a positive odd integer $x$ which is not divisible by three or seven.
Take the idempotent commutative quasigroup $(\textit{\textbf{Z}}_x, \circ)$, where $a \circ b = \frac{a+b}{2}$ for $a, b \in \textit{\textbf{Z}}_x$.
Define sets $\mathcal{B}_0$ and $\mathcal{B}_1$ of triples as follows:
\begin{align*}
\mathcal{B}_0 &= \{\{(a, 0), (a, 1), (a, 2)\}\ \vert \ a \in \textit{\textbf{Z}}_x\},\\
\mathcal{B}_1 &= \{\{(a, i), (b, i), (a \circ b, i+1)\}\ \vert \ a, b, \in \textit{\textbf{Z}}_x, a \not=b, i \in \textit{\textbf{Z}}_3\}.
\end{align*}
Then because $x$ is not divisible by seven and $(\textit{\textbf{Z}}_x, \circ)$ is an idempotent commutative quasigroup,
$(\textit{\textbf{Z}}_x\times\textit{\textbf{Z}}_3, \mathcal{B}_0\cup\mathcal{B}_1)$ forms an $S(2,3,3x)$ that contains no Pasch configurations
(see Brouwer \cite{ANTIPASCHBROUWER} and Doyen \cite{Doyen}).

It remains to show that the cyclic group of order $3x$ acts transitively on the points.
By assumption, $\gcd(x, 3) = 1$. Hence, it is obvious that these triples form the blocks of a cyclic Steiner $2$-design
over $\textit{\textbf{Z}}_x\times\textit{\textbf{Z}}_3 \cong \textit{\textbf{Z}}_{3x}$ as desired.
\end{proof}
The above cyclic $S(2,3,v)$s attain the upper bound on the even-freeness given in Proposition \ref{transitive}

\section{Recursive constructions}\label{rconst}
\noindent
In this section we give recursive constructions that generate cyclic $(k+1)$-even-free $S(2,k,v)$s from those on fewer points.
Applying the constructions to the resulting Steiner $2$-designs repeatedly may generate infinitely many cyclic $(k+1)$-even-free $S(2,k,v)$s.
We draw on the technique developed by Jimbo and Kuriki \cite{JK}, Colbourn and Colbourn \cite{CC}, and Grannell and Griggs \cite{GG} to ensure the cyclic property.
Unlike typical combinatorial construction techniques for set systems avoiding particular configurations,
where block size $k$ is usually limited to three or some small fixed number, our methods work for a wide range of $k$.

We begin with a simple observation about orbit structures of blocks of a cyclic Steiner $2$-design.
Taking $V$ as the set of nonnegative integers less than $v$,
the orbit $Orb_{\tau}(B)$ of a block $B$ of a cyclic $S(2,k,v)$ under the automorphism $\tau : i \mapsto i+1 \pmod{v}$ is of length either $v$ or $\frac{v}{k}$.
The latter occurs only when $v \equiv k \pmod{k(k-1)}$, where we must have $\{\frac{iv}{k} \ \vert \ 0 \leq i \leq k-1\}$ in the block set.
Let $\mathcal{S}$ be a system of representatives of all block orbits of a cyclic $S(2,k,v)$.
If $v \equiv 1 \pmod{k(k-1)}$, then every nonzero difference modulo $v$ appears exactly once as a difference between points in a block in $\mathcal{S}$.
If $v \equiv k \pmod{k(k-1)}$, then every nonzero difference modulo $v$ appears exactly once,
except for the multiples of $\frac{v}{k}$, which appear exactly $k$ times each.

Our first construction employs a special matrix.
A \textit{cyclic} $(v, k)$-\textit{difference matrix} is a $k \times v$ matrix $(a_{i, j})$, $0 \leq i \leq k-1$, $0 \leq j \leq v-1$, in which
each entry is an element of $\textit{\textbf{Z}}_v$ such that for each $0 \leq r < r' \leq k-1$,
the set $\{a_{r, j} - a_{r', j} \ \vert \ 0 \leq j \leq v-1\}$ of differences contains every element of $\textit{\textbf{Z}}_v$ exactly once.
The following is a simple and well-known result on the existence of a cyclic difference matrix.
\begin{theorem}[(Colbourn and Colbourn \cite{CC})]\label{dm}
If $\gcd(v, (k-1)!) = 1$, then there exists a cyclic $(v,k)$-difference matrix.
\end{theorem}
The following two theorems are particularly useful for our purpose.
\begin{theorem}[(Jungnickel \cite{Jungnickel})]\label{dmbibd1}
Let $k$ be a prime power and $v \equiv 1 \pmod{k(k-1)}$.
If there exists a cyclic $S(2,k,v)$, then there exists a cyclic $(v,k)$-difference matrix.
\end{theorem}
\begin{theorem}[(Ge \cite{Ge})]\label{dmproduct}
If there exist a cyclic $(v, k)$-difference matrix and cyclic $(w, k)$-difference matrix, then there exists a cyclic $(vw, k)$-difference matrix.
\end{theorem}
For a comprehensive list of existence results on cyclic difference matrices, the reader is referred to Colbourn and Dinitz \cite{HandbookCD} and references therein.

\begin{theorem}\label{main1}
Let $k$ be an even integer.
If there exist a cyclic $(k+1)$-even-free $S(2,k,v)$ with $v \equiv 1 \pmod{k(k-1)}$,
cyclic $(k+1)$-even-free $S(2,k,w)$ with $w \equiv 1 \pmod{k(k-1)}$, and cyclic $(w,k)$-difference matrix,
then there exists a cyclic $(k+1)$-even-free $S(2,k,vw)$.
\end{theorem}
\begin{proof}
Take two positive integers $v \equiv 1 \pmod{k(k-1)}$ and $w \equiv 1 \pmod{k(k-1)}$, where $k$ is even.
Take two sets $V = \{0, 1, \dots, v-1\}$ and $W = \{0, 1, \dots, w-1\}$ of nonnegative integers less than $v$ and $w$ respectively.
Let $(V, \mathcal{B})$ and $(W, \mathcal{C})$ be cyclic $(k+1)$-even-free Steiner $2$-designs of block size $k$ and order $v$ and $w$ respectively.
Take a system of representatives of all block orbits of $(V, \mathcal{B})$ and arbitrarily fix the order of points in each block,
so that we have a set $\mathcal{S}_{\mathcal{B}}$ of ordered sets,
where each block $\{x_i \ \vert \ 0 \leq i \leq k-1\} \in \mathcal{B}$ uniquely specifies an element $(x_0, x_1, \dots, x_{k-1}) \in \mathcal{S}_{\mathcal{B}}$.
Take a system $\mathcal{S}_{\mathcal{C}}$ of representatives of block orders of $(W, \mathcal{C})$.
Let also a $k \times w$ matrix $M = (a_{i, j})$, $0 \leq i \leq k-1$, $0 \leq j \leq w-1$, be a cyclic $(w,k)$-difference matrix.
Because adding an arbitrary element of $\textit{\textbf{Z}}_w$ to every element of a column of $M$ gives a cyclic difference matrix of the same parameters,
without loss of generality, we assume that $a_{0,j} = 0$ for every $j$.
Define sets $\mathcal{D}_0$ and $\mathcal{D}_1$ of $k$-tuples as follows:
\begin{align*}
\mathcal{D}_0 &= \bigcup_{j \in W}\left\{\{x_i + a_{i,j}v \ \vert \ 0 \leq  i \leq k-1\} \ \middle\vert \ (x_0, x_1, \dots, x_{k-1}) \in \mathcal{S}_{\mathcal{B}}\right\},\\
\mathcal{D}_1 &= \left\{\{y_iv \ \vert \ 0 \leq i \leq k-1\} \ \middle\vert \ \{y_i \ \vert \ 0 \leq i \leq k-1\} \in \mathcal{S}_{\mathcal{C}}\right\}.
\end{align*}
It is routine to check that $\mathcal{D}_0\cup\mathcal{D}_1$ forms a system of representatives of the block orbits of a cyclic $S(2,k,vw)$
on the point set $X = \{0, 1, \dots, vw-1\}$.
In fact, every difference not divisible by $v$ appears exactly once in a block in $\mathcal{D}_0$
while every nonzero difference divisible by $v$ appears exactly once in a block in $\mathcal{D}_1$.

It suffices to show that no generalized Pasch configuration occurs
in the set of blocks obtained by developing $\mathcal{D}_0\cup\mathcal{D}_1$ over $\textit{\textbf{Z}}_{vw}$.
We say that a block in an orbit represented by an element of $\mathcal{D}_0$ is of \textit{type} I.
A \textit{type} II block lies in an orbit represented by an element of $\mathcal{D}_1$.
Suppose to the contrary that the resulting cyclic $S(2,k,vw)$ contains a generalized Pasch configuration $\mathcal{P}$.
The rest of the proof is divided into three cases.

\textit{Case 1.}\quad All blocks in $\mathcal{P}$ are of type I.
Define the \textit{remainder block} of a block $B$ in $\mathcal{P}$ as the $k$-tuple or singleton obtained by taking $x \pmod{v}$ for all $x \in B$.
Trivially the remainder block of a type I block belongs to $\mathcal{B}$.
If the $k+1$ remainder blocks obtained from $\mathcal{P}$ fall into distinct blocks in $\mathcal{B}$,
we obtain a generalized Pasch configuration in $\mathcal{B}$, a contradiction.
Hence, there is a pair of blocks in $\mathcal{P}$ whose remainder blocks are the same.
Because any two blocks in $\mathcal{P}$ share one point and every pair of points appear exactly once in a Steiner $2$-design,
the remaining $k-1$ remainder blocks all fall into the same block in $\mathcal{B}$ as well.
Thus, the points in $\mathcal{P}$ can be grouped into $k$ categories according to which point in the remainder block they fall into.
By symmetry, every category has the same number of points in $\mathcal{P}$.
Hence, the total number of points in $\mathcal{P}$ is $ck$ for some integer $c$.
Because each pair of blocks in $\mathcal{P}$ intersect each other at one point,
the total number can also be expressed as $\sum_{i=0}^{k-1}(k-i)$.
Thus, we have $\frac{k(k+1)}{2} = ck$, which implies that $k = 2c-1$. This contradicts the assumption that $k$ is even.

\textit{Case 2.}\quad $\mathcal{P}$ contains a type II block.
A block of type II reduces to a single point in $V$ when each point is taken modulo $v$.
If there are two blocks of type II in $\mathcal{P}$, by considering remainder blocks,
the property of $\mathcal{P}$ that every pair of blocks intersect each other implies that all $k+1$ blocks are of type II.
By choosing an appropriate nonnegative integer $t$ less than $v$,
a block $B$ of type II can be written as $B = \{t + y_iv \ \vert \ 0 \leq i \leq k-1\}$.
Define the \textit{quotient block} of $B$ of type II as the underlying block $\{y_i \ \vert \ 0 \leq i \leq k-1\}$ in $\mathcal{C}$.
All the $k+1$ quotient blocks obtained from $\mathcal{P}$ are distinct
because otherwise we end up with identical blocks in $\mathcal{P}$.
However, this implies that $\mathcal{C}$ contains generalized Pasch, a contradiction.
Hence, $\mathcal{P}$ contains exactly one block $B$ of type II.
Translating $\mathcal{P}$ by adding an appropriate element of $\textit{\textbf{Z}}_{vw}$ to every point,
without loss of generality, we assume that $x \equiv 0 \pmod{v}$ for any $x \in B$.
Take a block $B'$ of type I in $\mathcal{P}$ and write its remainder block as $\{r_i \ \vert \ 0 \leq i \leq k-1\} \in \mathcal{B}$.
Because every remaining block intersects $B$ and $B'$ and every pair of points appear exactly once in a Steiner $2$-design,
the remainder blocks of the remaining $k-1$ blocks are also $\{r_i \ \vert \ 0 \leq i \leq k-1\}$.
Take a pair $B_c$ and $B_{c'}$ of blocks of type I in $\mathcal{P}$.
Without loss of generality, the two blocks can be written as
$B_c = \{r_i + a_{i,c}v \ \vert \ 0 \leq i \leq k-1\}$
and
$B_{c'} = \{r_i + a_{i,c'}v \ \vert \ 0 \leq i \leq k-1\}$.
Each pair of blocks in $\mathcal{P}$ share a point.
Let $r_s + a_{s,c} = r_s + a_{s,c'}$ be the shared point between $B_c$ and $B_{c'}$.
By assumption, we have $a_{0,j} = 0$ for any $j$. This implies that $a_{0,c} - a_{s,c} = a_{0,c'} - a_{s,c'}$.
If $s \not= 0$, we have the same difference twice between two rows, contradicting the definition of a cyclic difference matrix.
If $s = 0$, taking another pair of type I blocks in $\mathcal{P}$ leads to the same contradiction.
The proof is complete.
\end{proof}

With only a slight modification, the above argument still works if we assume that $w \equiv k \pmod{k(k-1)}$.
However, a cyclic difference matrix can not have more than two rows when $w$ is even (see Ge \cite{Ge}).

The construction can be extended to the case when the block size $k$ may be odd by exploiting a special small matrix.
An \textit{orthogonal array} OA$(t, s)$ is a $t \times s^2$ matrix with entries from a set $S$ of size $s$
such that in any two rows each ordered pair of symbols from $S$ appears exactly once.
A \textit{parallel class} of an OA$(t, s)$ is a set of $s$ columns in which each of the $s$ symbols appears in each row exactly once.
The following is a well-known classical result on the existence of orthogonal arrays.
\begin{proposition}\label{oa}
For any prime power $q$ there exists an \textup{OA}$(q, q)$ with a parallel class.
\end{proposition}
For other known results of orthogonal arrays, we refer the reader to Colbourn and Dinitz \cite{HandbookCD} and references therein.

\begin{theorem}\label{main2}
If there exist a cyclic $(k+1)$-even-free $S(2,k,v)$ with $v \equiv 1 \pmod{k(k-1)}$,
cyclic $(k+1)$-even-free $S(2,k,w)$ with $w \equiv 1 \pmod{k(k-1)}$, and \textup{OA}$(k,k)$ with a parallel class,
then there exists a cyclic $(k+1)$-even-free $S(2,k,vw)$.
\end{theorem}
\begin{proof}
If one of the two smaller Steiner $2$-designs is a trivial design of order one, then the statement is trivial.
We assume that the two ingredients are not trivial.
Let $v, w \equiv 1 \pmod{k(k-1)}$.
Take two sets $V = \{0, 1, \dots, v-1\}$ and $W = \{0, 1, \dots, w-1\}$ of nonnegative integers less than $v$ and $w$ respectively.
Let $(V, \mathcal{B})$ and $(W, \mathcal{C})$ be cyclic $(k+1)$-even-free Steiner $2$-designs of block size $k$ and order $v$ and $w$ respectively.
Take a system of representatives of all block orbits of $(V, \mathcal{B})$ and arbitrarily fix the order of points in each block,
so that we have a set $\mathcal{S}_{\mathcal{B}}$ of ordered sets.
Define a set $\mathcal{S}_{\mathcal{C}}$ of ordered sets the same way through a system of representatives of block orbits of $(W, \mathcal{C})$.
Let $k \times k^2$ matrix $K_S$ be an OA$(k,k)$ on $k$ symbols from set $S$ in which the last $k$ columns form a parallel class.
By renaming symbols, without loss of generality,
we assume that the $k$ entries of each column of the parallel class are the same.
For each $X \in \mathcal{S}_{\mathcal{C}}$,
define a $k \times k(k-1)-1$ matrix $L_X = (b_{i,j})$, $0 \leq i \leq k-1$, $0 \leq j \leq k(k-1)$,
by deleting the last $k$ columns from $K_X = (a_{i,j})$ on the $k$ symbols in the ordered set $X$.

Construct $k \times w$ matrix $M = (c_{i,j})$ by placing $L_X$ for all $X \in \mathcal{S}_{\mathcal{C}}$ side by side
and the $k$-dimensional zero vector as the first column.
It is straightforward to see that $M$ is a cyclic difference matrix over $\textit{\textit{Z}}_{w}$
in which the $k$ entries of each column except for the zero vector form a distinct block in $\mathcal{C}$.
As in the proof of Theorem \ref{main1}, define sets $\mathcal{D}_0$ and $\mathcal{D}_1$ of $k$-tuples as follows:
\begin{align*}
\mathcal{D}_0 &= \bigcup_{j \in W}\left\{\{x_i + c_{i,j}v \ \vert \ 0 \leq  i \leq k-1\} \ \middle\vert \ (x_0, x_1, \dots, x_{k-1}) \in \mathcal{S}_{\mathcal{B}}\right\},\\
\mathcal{D}_1 &= \left\{\{y_iv \ \vert \ 0 \leq i \leq k-1\} \ \middle\vert \ (y_0, y_1,\dots, y_{k-1}) \in \mathcal{S}_{\mathcal{C}}\right\}.
\end{align*}
We prove that the resulting cyclic $S(2,k,vw)$ over $\textit{\textbf{Z}}_{vw}$ contains no Pasch configurations if $k=3$
or generalized Pasch configurations if $k \geq 4$.
If $k$ is even, by adding $-c_{0,j}v$ modulo $vw$ to every point in the representative $\{x_i + c_{i,j}v \ \vert \ 0 \leq  i \leq k-1\}$ in $\mathcal{D}_0$ for each $j$,
the same argument as in the proof of Theorem \ref{main1} can be carried out to show the $(k+1)$-even-freeness of the $S(2,k,vw)$.
We assume that $k$ is odd.
As before, a block in an orbit represented by an element of $\mathcal{D}_0$ is of \textit{type} I while
a block in an orbit represented by an element of $\mathcal{D}_1$ is of \textit{type} II.
Any block  $B \in \mathcal{P}$ can be expressed in the form $\{t_i + u_iv \ \vert \ 0 \leq i \leq k-1\}$ with nonnegative integers $t_i$ less than $v$ and $u_i$ less than $w$.
As in the proof of the previous theorem, the remainder block of $B$ is the set $\{t_i \ \vert \ 0 \leq i \leq k-1\}$,
which is either a block in $\mathcal{B}$ or a singleton that consists of one point in $V$.
The quotient block of $B$ is $\{u_i \ \vert \ 0 \leq i \leq k-1\}$.
A quotient block is either a block in $\mathcal{C}$ or a singleton of a point in $W$.
We use these projections to prove $(k+1)$-even-freeness.
Suppose to the contrary that the resulting Steiner $2$-design contains a Pasch or generalized Pasch configuration $\mathcal{P}$.

\textit{Case 1.}\quad All blocks in $\mathcal{P}$ are of type I.
If the remainder blocks of the blocks in $\mathcal{P}$ are all distinct,
we obtain a Pasch or generalized Pasch configuration in $\mathcal{B}$, a contradiction.
Hence, we have a pair of blocks whose remainder blocks are the same.
Because every pair of blocks in $\mathcal{P}$ intersect each other and every pair of points in $V$ appear exactly once in a block in $\mathcal{B}$,
all $k+1$ blocks in $\mathcal{P}$ fall into the same remainder blocks.
Assume that there is a block in $\mathcal{P}$ whose quotient block is a singleton.
If there are two such blocks in $\mathcal{P}$, because every pair of blocks in $\mathcal{P}$ intersect each other,
all $k+1$ quotient blocks fall into the same singleton, leading to the contradiction that there are identical points among the $\frac{k(k+1)}{2}$ points in $\mathcal{P}$.
If there is exactly one block, say, $B$, whose quotient block is a singleton,
then the remaining $k$ quotient blocks fall into the same block in $\mathcal{C}$.
Delete all points in $B$ from the remaining $k$ blocks in $\mathcal{P}$ and form $k$ $(k-1)$-tuples.
The resulting $(k-1)$-tuples are isomorphic to Pasch or generalized Pasch on $k$ blocks of size $k-1$.
The points can be grouped into $k-1$ categories according to which point in the truncated quotient block they fall into.
However, because $k-1$ is even, the same double counting argument as in Case 1 of the proof of Theorem \ref{main1} leads to a contradiction.
Thus, all $k+1$ quotient blocks are blocks in $\mathcal{C}$.
If these quotient blocks are all distinct, we obtain a Pasch or generalized Pasch configuration in $\mathcal{C}$, a contradiction.
If there are two blocks whose quotient blocks are the same block in $\mathcal{C}$,
because every pair of blocks intersect each other in $\mathcal{P}$
and every pair of points appear exactly once in a block in a Steiner $2$-design,
the quotient blocks of all $k+1$ blocks in $\mathcal{P}$ are the same.
However, because each point in $\mathcal{P}$ appears exactly twice and the number of blocks in $\mathcal{P}$ is only $k+1$,
this implies that $L_X$ has a column in which the same symbol appears twice, a contradiction.

\textit{Case 2.}\quad $\mathcal{P}$ contains a block of type II.
If there are two blocks of type II, all blocks in $\mathcal{P}$ are of type II
while the quotient block of a block of type II is contained in $\mathcal{C}$.
Thus, it is either that $\mathcal{P}$ contains identical points or
that we have a Pasch or generalized Pasch configuration in $\mathcal{C}$, either of which is a contradiction.
Hence, we assume that $\mathcal{P}$ contains exactly one type II block, say, $B$.
Translating $\mathcal{P}$ by adding some element of $\textit{\textbf{Z}}_{vw}$ to each point,
without loss of generality, we assume that $x \equiv 0 \pmod{v}$ for any $x \in B$.
Take a block $B'$ of type I in $\mathcal{P}$ and write its remainder block as $\{r_i \ \vert \ 0 \leq i \leq k-1\} \in \mathcal{B}$, where $r_0 = 0$.
Because every remaining block intersects $B$ and $B'$ and every pair of points appear exactly once in a Steiner $2$-design,
the remainder blocks of the remaining $k-1$ blocks are also $\{r_i \ \vert \ 0 \leq i \leq k-1\}$.
Delete the points in $B$ from the remaining $k$ blocks and form a set $\mathcal{P}'$ of truncated blocks,
which is isomorphic to a Pasch or generalized Pasch configuration on $k$ blocks of size $k-1$.
The points in the $k$ truncated blocks are grouped into $k-1$ categories according to which nonzero point of $\{r_i \ \vert \ 1 \leq i \leq k-1\}$
they fall into when taken modulo $v$. However, because $k-1$ is even, it is impossible for each type of point to appear equal times
in a truncated block in $\mathcal{P}'$. The proof is complete.
\end{proof}

If the block size $k$ is an odd prime, the construction can be carried out under simpler, more relaxed conditions.
Note that in this case a $(k+1)$-even-free $S(2,k,v)$ is also $(k+2)$-even-free.
\begin{theorem}\label{main3}
Let $k$ be an odd prime. If there exist a cyclic $(k+2)$-even-free $S(2,k,v)$ with $v \equiv 1 \pmod{k(k-1)}$
and cyclic $(k+2)$-even-free $S(2,k,w)$,
then there exists a cyclic $(k+2)$-even-free $S(2,k,vw)$.
\end{theorem}
\begin{proof}
Let $k$ be an odd prime and $v$ a positive integer congruent to one modulo $k(k-1)$.
Let $(V, \mathcal{B})$ and $(W, \mathcal{C})$ be a cyclic $(k+2)$-even-free $S(2,k,v)$ and a cyclic $(k+2)$-even-free $(2,k,w)$ respectively.
For each $0 \leq a \leq k-1$, define a $k \times k$ matrix $K_a = (a_{i,j})$ by $a_{i,j} = ij + a \pmod{k}$.
Let $L$ be a $k \times k^2$ matrix obtained by placing $K_a$ for all $k$ side by side.
Because $k$ is an odd prime, it is easy to see that $L$ is an OA$(k,k)$ with a parallel class.
Hence, if $w \equiv 1 \pmod{k(k-1)}$, by Theorem \ref{main2} we obtain a cyclic $(k+2)$-even-free $S(2,k,vw)$ as desired.
We assume that $w \equiv k \pmod{k(k-1)}$.
Delete from $L$ the parallel class that consists of the $k$ columns $(i, i, \dots, i)^T$, $0 \leq i \leq k-1$.
As in the proof of Theorem \ref{main2},
take two sets $\mathcal{S}_{\mathcal{B}}$ and $\mathcal{S}_{\mathcal{C}}$ of ordered sets,
where the former is obtained by arbitrarily fixing the order of the points in the representative of each block orbit of $(V, \mathcal{B})$
while the latter is defined the same way by using $(W, \mathcal{C})$.
Let $S \in \mathcal{S}_{\mathcal{C}}$ be the one that corresponds to the representative of the orbit $Orb_{\textit{\textbf{Z}}_{w}}(\{\frac{iw}{k} \ \vert \ 0 \leq i \leq k-1\})$.
For each $X \in \mathcal{S}_{\mathcal{C}}\setminus S$ of the $\frac{w-k}{k(k-1)}$ ordered sets,
construct a copy $L_X = (b_{i,j})$ of the truncated OA by taking $X$ as its symbol set.
Place all $L_X$ side by side.
Place also the $k \times k$ matrix $\frac{w}{k}K_0$ on the symbol set $W$ obtained by multiplying each entry of $K_0$ by $\frac{w}{k}$.
It is easy to see that the resulting $k \times w$ matrix $M = (c_{i,j})$ forms a cyclic difference matrix over $\textit{\textbf{Z}}_{w}$.
Define sets $\mathcal{D}_0$ and $\mathcal{D}_1$ of $k$-tuples as follows:
\begin{align*}
\mathcal{D}_0 &= \bigcup_{j \in W}\left\{\{x_i + c_{i,j}v \ \vert \ 0 \leq  i \leq k-1\} \ \middle\vert \ (x_0, x_1, \dots, x_{k-1}) \in \mathcal{S}_{\mathcal{B}}\right\},\\
\mathcal{D}_1 &= \left\{\{y_iv \ \vert \ 0 \leq i \leq k-1\} \ \middle\vert \ (y_0, y_1, \dots, y_{k-1}) \in \mathcal{S}_{\mathcal{C}}\right\}.
\end{align*}
It is routine to check that $\mathcal{D}_0\cup\mathcal{D}_1$ forms a system of representatives of all block orbits of a cyclic $(k+2)$-even-free $S(2,k,vw)$.
\end{proof}

All three constructions given above require at least one Steiner $2$-design of block size $k$ to be of order congruent to $1$ modulo $k(k-1)$.
If $k=3$, the following recursive construction can take advantage of two ingredients both of which have orders congruent to $3$ modulo $6$:
\begin{theorem}\label{main4}
If there exist a cyclic $5$-even-free $S(2,3,3v)$ and cyclic $5$-even-free $S(2,3,3w)$,
then there exists a cyclic $5$-even-free $S(2,3,3vw)$.
\end{theorem}
\begin{proof}
Let $(V, \mathcal{B})$ and $(W, \mathcal{C})$ be a cyclic $5$-even-free $(2,3,3v)$ and a cyclic $5$-even-free $S(2,3,3w)$ respectively.
Take a system of representatives of block orbits of $(V, \mathcal{B})$ and arbitrarily fix the order of points in each representative,
so that we obtain a set $\mathcal{S}_{\mathcal{B}}$ of ordered triples.
Let $S \in \mathcal{S}_{\mathcal{B}}$ be the ordered triple that comes from the representative of $Orb_{\textit{\textbf{Z}}_{3v}}(\{0,v,2v\})$.
Let $\mathcal{S}_{\mathcal{C}}$ be a system of representatives of block orbits of $(W, \mathcal{C})$.
Define a pair, $\mathcal{D}_0$, $\mathcal{D}_1$, of sets of triples as follows:
\begin{align*}
\mathcal{D}_0 &= \bigcup_{i \in W}\left\{\{x, y + 3iv, z + 6iv \} \ \middle\vert \ (x, y, z) \in \mathcal{S}_{\mathcal{B}}\setminus S\right\},\\
\mathcal{D}_1 &= \left\{\{av, bv, cv\} \ \middle\vert \ \{a, b, c\} \in \mathcal{S}_{\mathcal{C}}\right\}.
\end{align*}
It is straightforward to check that $\mathcal{D}_0\cup\mathcal{D}_1$ forms a system of representatives of the block orbits of a cyclic $S(2,3,3vw)$.
A block is of type I if its orbit is represented by an element of $\mathcal{D}_0$. Otherwise it is of type II.
We prove that the resulting cyclic $S(2,3,3vw)$ avoids Pasch configurations.

\textit{Case 1.}\quad All blocks in $\mathcal{P}$ are of type I.
Reduce every point in $\mathcal{P}$ by taking modulo $v$.
If we obtain four distinct blocks, we end up with a Pasch configuration in $\mathcal{S}_{\mathcal{B}}\setminus S$, a contradiction.
If there are two blocks that are projected to the same block in $\mathcal{S}_{\mathcal{B}}\setminus S$,
then all four blocks in $\mathcal{P}$ fall into that same block.
Thus, without loss of generality, $\mathcal{P}$ can be written as
$\{\{x+t, y+t+3iv, z+t+6iv\}, \{x+t, y+t+3jv, z+t+6jv\}, \{x+t+kv, y+t+3iv, z+t+6jv\}, \{x+t+kv, y+t+3jv, z+t+6iv\}\}$,
where $(x,y,z) \in \mathcal{S}_{\mathcal{B}}\setminus S$ and $0 \leq i \not= j \leq w-1$.
Then, by the definition of $\mathcal{D}_0$, we have
\[k+2j \equiv 2i \pmod{w}\]
and
\[k+2i \equiv 2j \pmod{w}.\]
Because $w$ is odd, we have $i = j$, a contradiction.

\textit{Case 2.}\quad $\mathcal{P}$ contains a type II block.
If there are two blocks of type II in $\mathcal{P}$, by considering $x \pmod{v}$ for every point $x \in \mathcal{P}$,
all four blocks are of type II. However, projecting each point to an element of $\textit{\textbf{Z}}_{3w}$ through taking the quotient by $v$,
we obtain a Pasch configuration in $\mathcal{C}$, a contradiction.
Hence, we assume that $\mathcal{P}$ contains exactly one block of type II, say, $B$.
Translate $B$ to a block $B'$ by adding an appropriate element of $\textit{\textbf{Z}}_{3vw}$, so that $x \equiv 0 \pmod{v}$ for any $x \in B'$.
Reducing each point in $B'$ modulo $3v$, we obtain a list $L$ of size three.
$L$ is either $\{s,s,s\}$, $\{s,s,s+v\}$, $\{s,s,s-v\}$ or $\{s-v,s,s+v\}$ for some $0 \leq s \leq v-1$.
By considering which element in $\textit{\textbf{Z}}_{3v}$ each of the other three points shared by blocks of type I is reduced to when taken modulo $3v$,
the former three cases result in the contradiction that $\mathcal{S}_{\mathcal{B}}$ contains an ordered pair.
In the latter case, because $\{s-v,s,s+v\}$ is a block in $\mathcal{B}$, we obtain a Pasch configuration in $\mathcal{B}$.
This final contradiction completes the proof.
\end{proof}

\section{Concluding remarks}\label{conclusion}
\noindent
The constructions presented in the previous section generate infinitely many cyclic Steiner $2$-designs
with even-freeness higher than the trivial lower bound when combined with the known constructions and immediate results given in Section \ref{prep}.
Known examples found by computer searches are a good source of ingredients for our recursive constructions as well.
For instance, Colbourn, Mendelsohn, Rosa, and \v{S}ir\'{a}\v{n}\cite{CMRS} found a cyclic $S(2,3,v)$ avoiding Pasch configurations for
all admissible orders between $15$ and $97$ including $v = 21, 27, 63, 81$.
Applying Theorem \ref{main4} to these examples and those given in Proposition \ref{bose} substantially narrows the range of possible exceptions
of those orders for which a cyclic $5$-even-free $S(2,3,v)$ exists.

With all the infinitely many examples,
it may not be terribly surprising if one feels that it is plausible for a cyclic $(k+1)$-even-free $S(2,k,v)$ to exist
for all $v \equiv 1, k \pmod{k(k-1)}$ with only finitely many exceptions for each $k$.
However, proving or disproving this conjecture seems to be quite difficult.
In fact, it is already a formidable task to prove or disprove the existence of an $S(2,k,v)$ for all sufficiently large $v \equiv 1, k \pmod{k(k-1)}$
that are cyclic \textit{or} $(k+1)$-even-free.

Another ambitious goal is to understand which even configuration is avoidable in a Steiner $2$-design and which is not.
It appears that for any $k$ there exists a constant $v_k$ such that there exists a maximum $2$-$(v,k,1)$ packing containing no Pasch or generalized Pasch configurations.
However, we know little about larger even configurations except that
even configurations on $2k-1$ or fewer blocks can be simultaneously avoidable in projective geometry.
It is an interesting question whether nontrivial cyclic $S(2,k,v)$s with $k \geq 4$ can avoid all even configurations on $2k-1$ or fewer blocks.
If proven in the affirmative,
that would mean that cyclic automorphisms and extremely high even-freeness do not characterize the points and lines of projective geometry.

We took advantage of geometric structures shared among Pasch and generalized Pasch configurations across all block sizes
in order to make our construction techniques more general in terms of applicable block sizes.
While this allowed for arguments that work for a wide variety of block sizes,
a similar strategy, if exists for larger configurations in general, may require a careful classification of nonisomophic configurations in an $S(2,k,v)$
according to subtler geometric properties. Such a classification would be intriguing.

\providecommand{\bysame}{\leavevmode\hbox to3em{\hrulefill}\thinspace}
\providecommand{\MR}{\relax\ifhmode\unskip\space\fi MR }
\providecommand{\MRhref}[2]{%
  \href{http://www.ams.org/mathscinet-getitem?mr=#1}{#2}
}
\providecommand{\href}[2]{#2}

\affiliationone{Yuichiro Fujiwara\\
Mathematics 253-37\\
California Institute of Technology\\
Pasadena, California 91125\\
United States of America
   \email{yuichiro.fujiwara@caltech.edu}}
\end{document}